\documentclass[a4paper,10pt]{amsart}
\date{12 July 2009}
\title{Derived Equivalences Between Associative Deformations}

\author{Amnon Yekutieli}
\address{Department of  Mathematics,
Ben Gurion University,
Be'er Sheva 84105,
ISRAEL}
\email{amyekut@math.bgu.ac.il}
\thanks{{\em Mathematics Subject Classification} 2000.
Primary: 53D55; Secondary: 18E30, 16S80, 16D90.}
\keywords{Morita theory, tilting complexes, deformation quantization.}
\thanks{This research was supported by the Israel Science Foundation.}

\usepackage{amscd}
\usepackage{amssymb}
\usepackage[T1]{fontenc}
\usepackage{lmodern}

\newtheorem{thm}[equation]{Theorem}
\newtheorem{cor}[equation]{Corollary}
\newtheorem{prop}[equation]{Proposition}
\newtheorem{lem}[equation]{Lemma}
\theoremstyle{definition}
\newtheorem{dfn}[equation]{Definition}
\newtheorem{rem}[equation]{Remark}
\newtheorem{exa}[equation]{Example}

\newtheorem{setup}[equation]{Setup}

\numberwithin{equation}{section}

\newcommand{\xar}{\xrightarrow}
\newcommand{\opn}{\operatorname}
\newcommand{\cat}[1]{\operatorname{\mathsf{#1}}}

\newcommand{\mfrak}[1]{\mathfrak{#1}}

\newcommand{\mrm}[1]{\mathrm{#1}}
\newcommand{\mbb}[1]{\mathbb{#1}}

\newcommand{\tup}[1]{\textup{#1}}

\newcommand{\til}[1]{\tilde{#1}}
\newcommand{\what}[1]{\hat{#1}}
\newcommand{\hatotimes}[1]{\, \what{\otimes}_{#1} \,}

\renewcommand{\k}{\Bbbk}
\newcommand{\K}{\mbb{K}}
\newcommand{\R}{\mbb{R}}

\newcommand{\Z}{\mbb{Z}}
\newcommand{\p}{\mfrak{p}}
\newcommand{\q}{\mfrak{q}}
\newcommand{\m}{\mfrak{m}}
\renewcommand{\a}{\mfrak{a}}

\renewcommand{\d}{\mrm{d}}

\begin{document}

\begin{abstract}
We prove that if two associative deformations (parameterized by
the same complete local ring) are derived Morita equivalent, then they are
Morita equivalent (in the classical sense).
\end{abstract}

\maketitle

\setcounter{section}{-1}
\section{Introduction}

Let $\K$ be a commutative ring, and let $A$ and $B$ be associative unital
$\K$-algebras. We denote by $\cat{Mod} A$ and $\cat{Mod} B$ the corresponding
categories of left modules. One says that $A$ and $B$ are {\em Morita equivalent
relative to $\K$} (in the classical sense) if there is a $\K$-linear equivalence
of categories
$\cat{Mod} A \to \cat{Mod} B$.

Let $\cat{D}^{\mrm{b}}(\cat{Mod} A)$ denote the bounded derived category of 
complexes of left $A$-modules. This is a $\K$-linear triangulated category.
If there is a $\K$-linear equivalence of triangulated categories
$\cat{D}^{\mrm{b}}(\cat{Mod} A) \to \cat{D}^{\mrm{b}}(\cat{Mod} B)$,
then one says that $A$ and $B$ are {\em derived Morita equivalent
relative to $\K$}.

There are plenty of examples of pairs of algebras that are derived Morita
equivalent, but are not Morita equivalent in the classical sense.

Now suppose $\K$ is a complete noetherian local ring, with maximal
ideal $\m$ and residue field $\k$.  
Let $A$ be a flat $\m$-adically complete $\K$-algebra, such that the
$\k$-algebra $\bar{A} := \k \otimes_\K A$ is commutative. We then say that $A$
is an {\em associative $\K$-deformation of $\bar{A}$}; see \cite{Ye3}.

The most important example of an associative deformation is when 
$\k = \R$; $\bar{A} = \mrm{C}^{\infty}(X)$, the $\R$-algebra of smooth functions
on a differentiable manifold $X$; $\K = \R[[\hbar]]$, the ring of formal power
series in the variable $\hbar$; 
and $A = \bar{A}[[\hbar]]$. In this case the multiplication in $A$ is called a
{\em star product}.

Let us assume that $A$ and $B$ are associative $\K$-deformations, and moreover
the commutative rings $\bar{A}$ and $\bar{B}$ have connected prime spectra
(i.e.\ they have no nontrivial idempotents). 
The main result of the paper (Theorem \ref{thm:5}) says that if
$T$ is a {\em two-sided tilting complex over $B$-$A$ relative to $\K$}, then 
$T \cong P[n]$ for some invertible bimodule $P$ and integer $n$. 
(Tilting complexes and their properties are recalled in Section 1.)
A direct consequence (Corollary \ref{cor:2}) is that if $A$ and
$B$ are derived Morita equivalent, then they are Morita equivalent in the
classical sense.

\medskip \noindent
{\em Acknowledgments.}
The problem was brought to my attention by H. Bursztyn and S. Waldmann
during the conference ``Algebraic Analysis and Deformation Quantization'' in
Scalea, Italy in June 2009. I wish to thank Bursztyn and Waldmann for explaining
their work to me, and the organizers of the conference for providing the 
background for this interaction.

\section{Base Change for Tilting Complexes}

In this section we recall some facts about two-sided tilting complexes, and
also prove one new theorem.
Throughout this section $\K$ is a commutative ring. 
By ``$\K$-algebra'' we mean an associative unital algebra; i.e.\ a ring $A$,
with center $\mrm{Z}(A)$, together with a ring homomorphism
$\K \to \mrm{Z}(A)$.

For a $\K$-algebra $A$ we denote by $A^{\mrm{op}}$ the opposite algebra, namely
with reverse multiplication. We view right $A$-modules as left
$A^{\mrm{op}}$-modules. Let $B$ be some other $\K$-algebra. 
By {\em $B$-$A$-bimodule relative to $\K$} we mean a 
$\K$-central $B$-$A$-bimodule. We view $B$-$A$-bimodules relative to $\K$
as left $B \otimes_{\K} A^{\mrm{op}}$ -modules.

The category of left $A$-modules is denoted by $\cat{Mod} A$.
This is a $\K$-linear abelian category.
Classical Morita theory says that any $\K$-linear equivalence
$\cat{Mod} A \to \cat{Mod} B$
is of the form $P \otimes_A -$, where $P$ is some invertible
$B$-$A$-bimodule relative to $\K$.

The derived category of $\cat{Mod} A$ is
$\cat{D}(\cat{Mod} A)$. 
This is a $\K$-linear triangulated category.
We follow the conventions of \cite{RD} on derived categories.
For instance, $\cat{D}^{\mrm{b}}(\cat{Mod} A)$ is the full subcategory of
$\cat{D}(\cat{Mod} A)$ consisting of bounded complexes.

Here is a definition from Rickard's paper \cite{Ri1}.

\begin{dfn}
Let $A$ and $B$ be $\K$-algebras. If there exists a $\K$-linear equivalence of
triangulated categories
$\cat{D}^{\mrm{b}}(\cat{Mod} A) \to \cat{D}^{\mrm{b}}(\cat{Mod} B)$
then we say that $A$ and $B$ are {\em derived Morita equivalent relative to
$\K$}.
\end{dfn}

Now assume that $A$ is {\em flat} over $\K$.
Since $A \otimes_{\K} B$ is flat over $B$, it follows
that the forgetful functor
$\cat{Mod} A \otimes_{\K} B \to \cat{Mod} B$
sends flat modules to flat modules. 

Given three $\K$-algebras $A, B, C$, and complexes
$M \in \cat{D}^-(\cat{Mod} A \otimes_{\K} B^{\mrm{op}})$
and
$N \in \cat{D}^-(\cat{Mod} B \otimes_{\K} C^{\mrm{op}})$,
and assuming $A$ is flat over $\K$,
the derived tensor product
\[ M \otimes^{\mrm{L}}_B N \in \cat{D}^-(\cat{Mod} A \otimes_{\K} C^{\mrm{op}})
\]
can be defined as follows: choose a quasi-isomorphism
$P \to M$ with $P$ a bounded above complex of projective 
$A \otimes_{\K} B^{\mrm{op}}$ -modules. Then $P$ is a bounded above complex of
flat $B^{\mrm{op}}$-modules, and we take
\[ M \otimes^{\mrm{L}}_B N  := P \otimes_B N . \]
This operation is functorial in $M$ and $N$.
As usual the requirements can be relaxed: it is enough to resolve $M$ by a
bounded above complex $P$ of
bimodules that are flat over $B^{\mrm{op}}$. If $C$ is flat over $\K$ then we
can resolve $N$ instead of $M$.
The derived tensor product 
$M \otimes^{\mrm{L}}_B N$
is ``indifferent'' to the algebras $A$ and $C$: we can forget them before or
after calculating $M \otimes^{\mrm{L}}_B N$, and get the same answer in
$\cat{D}^-(\cat{Mod} \K)$.

We record the following useful technical results.

\begin{lem}[Projective truncation trick] \label{lem:3}
Let $M \in \cat{D}(\cat{Mod} A)$ and let $i_0$ be an integer.
Suppose that $\mrm{H}^i M = 0$ for all $i > i_0$, and $P := \mrm{H}^{i_0} M$
is a projective $A$-module. Then there is an isomorphism
$M \cong P[-i_0] \oplus N$ in $\cat{D}(\cat{Mod} A)$, where 
$N$ is a complex satisfying $N^i = 0$ for all $i \geq i_0$.
\end{lem}

\begin{proof}
By the usual truncation trick (cf.\ \cite[Section I.7]{RD}) we can assume that 
$M^i = 0$ for all $i > i_0$.  Hence we get an exact sequence
$M^{i_0 - 1} \xar{\d} M^{i_0} \to P \to 0$. 
But $P$ is projective, and therefore 
$M^{i_0} \cong P \oplus \d(M^{i_0 - 1})$. 
Define $N^{i_0 - 1} := \opn{Ker}(\d) \subset M^{i_0 - 1}$
and $N^i := M^i$ for $i < i_0 - 1$. 
\end{proof}

Recall that a complex
$M \in \cat{D}(\cat{Mod} A)$ is called {\em perfect} if it is isomorphic to
bounded complex of finitely generated projective modules.
We denote by \linebreak $\cat{D}(\cat{Mod} A)_{\mrm{perf}}$ the full subcategory
of perfect complexes.

\begin{lem} \label{lem:1}
Let $M \in \cat{D}(\cat{Mod} A)_{\mrm{perf}}$ and let $i_0$ be
an integer. If $\mrm{H}^i M = 0$ for all $i > i_0$, then the $A$-module 
$N := \mrm{H}^{i_0} M$ is finitely presented.
\end{lem}

\begin{proof}
This is a bit stronger then \cite[Lemma 1.1(2)]{Ye1}.
By truncation reasons we can assume that 
$M \cong P$, where $P$ is a bounded complex of finitely generated projective
$A$-modules, and $P^i = 0$ for $i > i_0$. So we get an exact sequence 
$P^{i_0 - 1} \to P^{i_0} \to N \to 0$. Suppose $P^{i_0}$ is a direct summand of
$A^{r}$ (the free module of rank $r$), and 
$P^{i_0 - 1}$ is a direct summand of $A^{s}$. Then be rearranging terms we get
an exact sequence
$A^{r+s} \to A^r \to N \to 0$.
\end{proof}

\begin{lem}[K\"unneth trick] \label{lem:2}
Let $A$ be a $\K$-algebra, let
$M \in \cat{D}^-(\cat{Mod} A^{\mrm{op}})$
and let
$N \in \cat{D}^-(\cat{Mod} A)$.
Let $i_0, j_0 \in \Z$ be such that 
$\mrm{H}^{i} M = 0$ and 
$\mrm{H}^{j} N = 0$
for all $i > i_0$ and $j > j_0$. Then 
\[ (\mrm{H}^{i_0} M) \otimes_A (\mrm{H}^{j_0} N) \cong
\mrm{H}^{i_0 + j_0} (M \otimes^{\mrm{L}}_A N)  \]
as $\K$-modules.
\end{lem}

\begin{proof}
See \cite[Lemma 2.1]{Ye1}. 
\end{proof}

The next definition is from \cite{Ri2}.

\begin{dfn}
Let $A$ and $B$ be flat $\K$-algebras. A {\em two-sided tilting complex over
$B$-$A$ relative to $\K$} is a complex
$T \in \cat{D}^{\mrm{b}}(\cat{Mod} B \otimes_{\K} A^{\mrm{op}})$
with the following property: 
\begin{itemize}
\item[($\ast$)] there exists a complex 
$S \in \cat{D}^{\mrm{b}}(\cat{Mod} A \otimes_{\K} B^{\mrm{op}})$,
and isomorphisms
$S \otimes^{\mrm{L}}_B T \cong A$
and
$T \otimes^{\mrm{L}}_A S \cong B$
in 
$\cat{D}^{\mrm{b}}(\cat{Mod} A \otimes_{\K} A^{\mrm{op}})$
and
$\cat{D}^{\mrm{b}}(\cat{Mod} B \otimes_{\K} B^{\mrm{op}})$
respectively. 
\end{itemize}

The complex $S$  is called an {\em inverse of $T$}.

In case $B = A$ we say that $T$ is a 
{\em two-sided tilting complex over $A$ relative to $\K$}.
\end{dfn}

The inverse $S$ in the definition is unique up to isomorphism in 
$\cat{D}^{\mrm{b}}(\cat{Mod} A \otimes_{\K} B^{\mrm{op}})$.
Of course $S$ is a two-sided tilting complex over
$A$-$B$ relative to $\K$.

A two-sided tilting complex $T$ induces a $\K$-linear equivalence of
triangulated categories
\[ T \otimes^{\mrm{L}}_A - : 
\cat{D}(\cat{Mod} A) \to \cat{D}(\cat{Mod} B) . \]
This functor restricts to equivalences 
\[ \cat{D}^{\star}(\cat{Mod} A) \to 
\cat{D}^{\star}(\cat{Mod} B) , \]
where $\star$ is either $+, -$ or $\mrm{b}$; and also to an
equivalence
\[ \cat{D}(\cat{Mod} A)_{\mrm{perf}} \to 
\cat{D}(\cat{Mod} B)_{\mrm{perf}} . \]
 See \cite{Ri2} or \cite[Corollary 1.6(4)]{Ye1}.

Conversely we have the next important result, due to Rickard \cite{Ri2}.
For alternative proofs see \cite{Ke1} or \cite[Corollary 1.9]{Ye1}.

\begin{thm}[Rickard] \label{thm:1}
Let $A$ and $B$ be flat $\K$-algebras that are derived Morita equivalent
relative to $\K$. Then there exists a two-sided tilting complex over
$B$-$A$ relative to $\K$.
\end{thm}

\begin{rem} 
Suppose 
$F : \cat{D}(\cat{Mod} A) \to \cat{D}(\cat{Mod} B)$
is a $\K$-linear equivalence of triangulated categories. Then $F$ restricts to
an equivalence between the subcategories of perfect complexes (cf.\ \cite{Ke2}).
This implies that $F$ has finite cohomological dimension
(bounded by the amplitude of $\mrm{H}\, F (A)$). Hence $F$
restricts to an equivalence between the bounded derived categories -- i.e.\ a
derived Morita equivalence.
\end{rem}

\begin{rem}
In our paper \cite{Ye1} the base ring $\K$ is taken to be a field. 
However the results in Sections 1-3 of that paper hold for any commutative base
ring $\K$, as long as the $\K$-algebras in question are {\em flat}. 

It is possible to remove even the flatness condition, at the price of working
with DG algebras. Here is how to do it: choose a DG $\K$-algebra 
$\til{A}$ such that $\til{A}^i = 0$ for $i > 0$ and every $\til{A}^i$ flat as
$\K$-module, with a DG algebra quasi-isomorphism $\til{A} \to A$. 
We call $\til{A} \to A$ a flat DG algebra resolution of $A$ relative to $\K$. 
This can be done (cf.\ \cite[Section 1]{YZ} for commutative $\K$-algebras). 
Likewise choose a flat DG algebra resolution $\til{B} \to B$. 

Let $\til{\cat{D}}(\cat{DGMod} \til{A})^{\mrm{b}}$ be the derived category 
of DG $\til{A}$-modules with bounded cohomologies. 
It is known (cf.\ \cite[Proposition 1.4]{YZ}) that the restriction of scalars
functor
$\cat{D}^{\mrm{b}}(\cat{Mod} A) \to 
\til{\cat{D}}(\cat{DGMod} \til{A})^{\mrm{b}}$
is an equivalence. Therefore a $\K$-linear equivalence 
$\cat{D}^{\mrm{b}}(\cat{Mod} A) \to 
\cat{D}^{\mrm{b}}(\cat{Mod} B)$
is the same as a $\K$-linear equivalence 
$\til{\cat{D}}(\cat{DGMod} \til{A})^{\mrm{b}} \to 
\til{\cat{D}}(\cat{DGMod} \til{B})^{\mrm{b}}$.
Now the proof of \cite[Theorem 1.8]{Ye1} shows that there is a complex
$T \in \til{\cat{D}}(\cat{DGMod} \til{B} 
\otimes_{\K} \til{A}^{\mrm{op}})^{\mrm{b}}$
which is two-sided tilting. 

A different choice of flat DG algebra resolutions $\til{A}\to A$ and
$\til{B} \to B$
will give rise to an equivalent triangulated category
$\til{\cat{D}}(\cat{DGMod} \til{B} 
\otimes_{\K} \til{A}^{\mrm{op}})^{\mrm{b}}$. 
In this sense two-sided tilting
complexes are
independent of the resolutions.
\end{rem}

See Remark \ref{rem:1} for the history of the next theorem.

\begin{thm} \label{thm:2}
Let $A$ and $B$ be flat $\K$-algebras. Assume $A$ is commutative with connected
spectrum. Let $T$ be a two-sided tilting complex over $B$-$A$ relative to $\K$.
Then there is an isomorphism
\[ T \cong P[n] \]
in 
$\cat{D}^{\mrm{b}}(\cat{Mod} B \otimes_{\K} A^{\mrm{op}})$
for some invertible $B$-$A$-bimodule $P$ and integer $n$.
\end{thm}

\begin{proof}
We may assume that $A \neq 0$, so that $T \neq 0$.
The complex $T$ is perfect over $B$ and over $A^{\mrm{op}}$
(cf.\ \cite[Theorem 1.6]{Ye1}).
As in \cite[Proposition 2.4]{Ye1} the complex $T$ induces a $\K$-algebra
isomorphism $A \cong \mrm{Z}(B)$.

Let
\[ n := - \opn{sup} \{ i \mid \mrm{H}^i T \neq 0 \} , \]
and let $P := \mrm{H}^{-n} T$. This is a $B$-$A$-bimodule. 
By Lemma \ref{lem:1}, $P$ is finitely presented as right $A$-module.

For a prime $\p \in \opn{Spec} A$, with  corresponding local ring
$A_{\p}$, we write $P_{\p} := P \otimes_A A_{\p}$.
Define $Y \subset \opn{Spec} A$ to be the support of $P$, i.e.\ 
\[ Y := \{ \p \in \opn{Spec} A \mid P_{\p} \neq 0 \} . \]
Since $P$ is finitely generated it follows that $Y$ is a closed subset of 
$\opn{Spec} A$.

Take any prime $\p \in Y$, and let $B_{\p} :=  B \otimes_A A_{\p}$.
Then, by \cite[Lemma 2.5]{Ye1}, the complex
\[ T_{\p} := B_{\p} \otimes_B T \otimes_A A_{\p} \in 
\cat{D}^{\mrm{b}}(\cat{Mod} B_{\p} \otimes_{\K} A_{\p}^{\mrm{op}})
\]
is a two-sided tilting complex over $B_{\p}$-$A_{\p}$.
Since
\[ \mrm{H}^{-n} T_{\p} \cong P_{\p} \neq 0 , \]
\cite[Theorem 2.3]{Ye1} implies that 
\begin{equation} \label{eqn:1}
T_{\p} \cong P_{\p}[n] \in \cat{D}^{\mrm{b}}(\cat{Mod} B_{\p} \otimes_{\K}
A_{\p}^{\mrm{op}}) .
\end{equation}
Thus $P_{\p}$ is an invertible $B_{\p}$-$A_{\p}$-bimodule. This implies that 
$P_{\p}$ is a free $A_{\p}$-module, of rank $r > 0$.
According to \cite[Section II.5.1,  Corollary]{CA} there is an open
neighborhood $U$ of $\p$ in $\opn{Spec} A$ on which $P$ is free of rank $r$. 
In particular $P_{\q} \neq 0$ for all $\q \in U$.
Therefore $U \subset Y$. 

The conclusion is that $Y$ is also open in $\opn{Spec} A$.
Since $\opn{Spec} A$ is connected it follows that 
$Y = \opn{Spec} A$.
Another conclusion is that $P$ is projective as $A$-module --
see \cite[Section II.5.2, Theorem 1]{CA}.

Going back to equation (\ref{eqn:1}) we see that 
$(\mrm{H}^i T)_{\p} \cong \mrm{H}^i T_{\p} = 0$
for all $i \neq -n$. Therefore 
$\mrm{H}^i T = 0$ for $i \neq -n$. By truncation we get an isomorphism
$T \cong P[n]$ in 
$\cat{D}^{\mrm{b}}(\cat{Mod} B \otimes_{\K} A^{\mrm{op}})$.
Finally by \cite[Proposition 2.2]{Ye1} the $B$-$A$-bimodule $P$ is invertible.
\end{proof}

\begin{rem} \label{rem:1}
Theorem \ref{thm:2} (for a field $\K$) is \cite[Theorem 2.6]{Ye1}.
However the proof there is only correct when $A$ is noetherian (the hidden
assumption is that $\opn{Spec} A$ is a noetherian topological space). 

The same result was proved independently (and pretty much simultaneously, i.e.\
circa 1997) by Rouquier and Zimmermann \cite{RZ}. 
\end{rem}

\begin{cor}
Let $A$ and $B$ be flat $\K$-algebras with $A$ commutative.
If $A$ and $B$ are derived Morita equivalent relative to $\K$, then they are
Morita equivalent relative to $\K$.
\end{cor}

\begin{proof}
Use the first paragraph in the proof of \cite[Theorem 2.6]{Ye1} to pass to the
case when $\opn{Spec} A$ is connected, and then apply Theorem \ref{thm:2}.
\end{proof}

We denote by $\opn{Pic}_{\K}(A)$ the noncommutative Picard group of $A$,
consisting of isomorphism classes of invertible
$A$-$A$-bimodules relative to $\K$. The operation is $- \otimes_A -$.
Here is a definition from \cite{Ye1} extending this notion to the derived
setting:

\begin{dfn}
Let $A$ be a flat $\K$-algebra. The {\em derived Picard group of $A$ relative
to $\K$} is 
\[ \opn{DPic}_{\K}(A) := 
\frac{ \{ \tup{two-sided tilting complexes over $A$ relative to $\K$} \} }
{\tup{isomorphism}} , \]
where isomorphism is in 
$\cat{D}^{\mrm{b}}(\cat{Mod} A \otimes_{\K} A^{\mrm{op}})$.
The operation is $- \otimes^{\mrm{L}}_A -$, and the unit element is the
bimodule $A$. 
\end{dfn}

There is a canonical injective group homomorphism
\[ \opn{Pic}_{\K}(A) \times \Z \to  \opn{DPic}_{\K}(A) . \]
It formula is $(P, n) \mapsto P[n]$.

\begin{rem}
When $A$ is either local, or commutative with connected spectrum,
the homomorphism above is in fact bijective. 
On the other hand, if $A$ is the algebra of upper triangular 
$n \times n$ matrices over $\K$ ($n > 0$, $\K$ a field), then the 
bimodule 
$A^* := \opn{Hom}_{\K}(A, \K)$ is a two-sided tilting complex that does not
belong to $\opn{Pic}_{\K}(A) \times \Z$. 
This is a sort of ``Calabi-Yau'' phenomenon. See \cite{Ye1} for details. 
\end{rem}

Let $A$ and $B$ be $\K$-algebras, and let $P$ be an invertible
$B$-$A$-bimodule relative to $\K$.
Let $\K'$ be any commutative $\K$-algebra, and define 
$A' := \K' \otimes_{\K} A$, $B' := \K' \otimes_{\K} B$.
and $P' := \K' \otimes_{\K} P$.
Then $P'$ is an invertible $B'$-$A'$-bimodule relative to $\K'$.
When we take $B = A$ this fact gives rise to a group homomorphism
\[ \opn{Pic}_{\K}(A) \to \opn{Pic}_{\K'}(A') . \]

For the derived version we need flatness. The next theorem is the only new
result in this section of the paper. 

\begin{thm} \label{thm:3}
Let $A, B, C$ be flat $\K$-algebras, and let $\K'$ be a commutative
$\K$-algebra. Define $A' := \K' \otimes_{\K} A$, $B' := \K' \otimes_{\K} B$
and $C' := \K' \otimes_{\K} C$.
Given complexes
\[ M \in \cat{D}^{-}(\cat{Mod} A \otimes_{\K} B^{\mrm{op}}) \]
and
\[ N \in \cat{D}^{-}(\cat{Mod} B \otimes_{\K} C^{\mrm{op}}) , \]
let us define
\[ M' := \K' \otimes^{\mrm{L}}_{\K} M \in 
\cat{D}^{-}(\cat{Mod} A' \otimes_{\K'} {B'}^{\, \mrm{op}}) \]
and
\[ N' := \K' \otimes^{\mrm{L}}_{\K} N \in 
\cat{D}^{-}(\cat{Mod} B' \otimes_{\K'} {C'}^{\, \mrm{op}}) . \]
Then there is an isomorphism
\[ M' \otimes^{\mrm{L}}_{B'} N' \cong
\K' \otimes^{\mrm{L}}_{\K} (M \otimes^{\mrm{L}}_{B} N) \]
in 
$\cat{D}^{-}(\cat{Mod} A \otimes_{\K} C^{\mrm{op}})$,
functorial in $M$ and $N$.
\end{thm}

\begin{proof}
First let us observe that $A \otimes_{\K} B^{\mrm{op}}$ is a flat
$\K$-algebra, and 
\[ A' \otimes_{\K'} {B'}^{\, \mrm{op}} \cong 
\K' \otimes_{\K} (A \otimes_{\K} B^{\mrm{op}}) \]
as $\K'$-algebras.

Choose an isomorphism
$M \cong P$ in 
$\cat{D}^{-}(\cat{Mod} A \otimes_{\K} B^{\mrm{op}})$,
where each $P^i$ is projective over $A \otimes_{\K} B^{\mrm{op}}$.
Then 
\[ M' \cong \K' \otimes_{\K} P \in
\cat{D}^{-}(\cat{Mod} A' \otimes_{\K'} {B'}^{\, \mrm{op}}) , \]
and each $\K' \otimes_{\K} P^i$ is flat over $A'$ and over
${B'}^{\, \mrm{op}}$.

Similarly let us choose an isomorphism
$N \cong Q$ in 
$\cat{D}^{-}(\cat{Mod} B \otimes_{\K} C^{\mrm{op}})$;
so
$N' \cong \K' \otimes_{\K} Q$.

Now
\[ M' \otimes^{\mrm{L}}_{B'} N' \cong
(\K' \otimes_{\K} P) \otimes_{B'} (\K' \otimes_{\K} Q) \]
in
$\cat{D}^{-}(\cat{Mod} A' \otimes_{\K} {C'}^{\, \mrm{op}})$.
There is a  canonical isomorphism
\[ (\K' \otimes_{\K} P) \otimes_{B'} (\K' \otimes_{\K} Q) \cong
\K' \otimes_{\K} (P \otimes_{B} Q) \]
as complexes of $A' \otimes_{\K} {C'}^{\, \mrm{op}}$ -modules; and therefore
this is also an isomorphism also in 
$\cat{D}^{-}(\cat{Mod} A' \otimes_{\K} {C'}^{\, \mrm{op}})$.

Next we have
\[ M \otimes^{\mrm{L}}_{B} N \cong P \otimes_{B} Q \]
in $\cat{D}^{-}(\cat{Mod} A \otimes_{\K} {C}^{\mrm{op}})$.
But since $P \otimes_{B} Q$ is a complex of flat $\K$-modules, we also have
\[ \K' \otimes^{\mrm{L}}_{\K} (M \otimes^{\mrm{L}}_{B} N) \cong 
\K' \otimes_{\K} (P \otimes_{B} Q)  \]
in 
$\cat{D}^{-}(\cat{Mod} A' \otimes_{\K} {C'}^{\, \mrm{op}})$.
\end{proof}

\begin{cor} \label{cor:1}
Let $A$ and $B$ be flat $\K$-algebras, and let $\K'$ be a commutative
$\K$-algebra. Define $A' := \K' \otimes_{\K} A$ and $B' := \K' \otimes_{\K} B$.
Suppose $T$ is a two-sided tilting complex over $B$-$A$ relative to $\K$,
with inverse $S$.
Define 
\[ T' := \K' \otimes^{\mrm{L}}_{\K} T \in 
\cat{D}^{\mrm{b}}(\cat{Mod} B' \otimes_{\K'} {A'}^{\, \mrm{op}}) \]
and
\[ S' := \K' \otimes^{\mrm{L}}_{\K} S \in 
\cat{D}^{\mrm{b}}(\cat{Mod} A' \otimes_{\K'} {B'}^{\, \mrm{op}}) . \]
Then $T'$ is a is a two-sided tilting complex over $B'$-$A'$ relative to
$\K'$, with inverse $S'$.
\end{cor}

\begin{proof}
By the theorem we have
\[ T' \otimes^{\mrm{L}}_{A'} S' \cong 
\K' \otimes^{\mrm{L}}_{\K} (T \otimes^{\mrm{L}}_{A} S) \cong
\K' \otimes^{\mrm{L}}_{\K} B \cong B' \]
in
$\cat{D}^{\mrm{b}}(\cat{Mod} B' \otimes_{\K'} {B'}^{\mrm{op}})$; 
and similarly
$S' \otimes^{\mrm{L}}_{B'} T' \cong A'$. 
\end{proof}

\begin{cor}
Let $A$ be a flat $\K$-algebra, and let $\K'$ be a commutative
$\K$-algebra. Define $A' := \K' \otimes_{\K} A$. Then
the formula 
$T \mapsto \K' \otimes^{\mrm{L}}_{\K} T$
defines a group homomorphism
\[ \opn{DPic}_{\K}(A) \to \opn{DPic}_{\K'}(A') . \]
\end{cor}

\begin{proof}
Immediate from the previous corollary.
\end{proof}

\section{Associative Deformations}

In this section we keep the following setup:

\begin{setup} \label{setup:1}
$\K$ is a complete local noetherian commutative ring, with
maximal ideal $\m$ and residue field $\k = \K / \m$.
\end{setup}

Let $M$ be a $\K$-module. Its $\m$-adic completion is the $\K$-module
\[ \what{M} := \lim_{\leftarrow i}\, M / \m^i M . \]
Recall that $M$ is called
{\em $\m$-adically complete} 
(some texts, e.g.\ \cite{CA}, use the term ``separated and complete'')
if the canonical homomorphism $M \to \what{M}$ is bijective. 
Every finitely generated $\K$-module is complete; but this is not true for
infinitely generated modules. For instance, if $N$ is a free $\K$-module of
infinite rank, and if the ideal $\m$ is not nilpotent, then the canonical
homomorphism $N \to \what{N}$ in injective but not surjective. 
Still in this instance the induced homomorphism
$\k \otimes_{\K} N \to \k \otimes_{\K} \what{N}$
is bijective. See \cite[Theorem 1.12]{Ye2}.

In \cite[Corollary 2.12]{Ye2} we prove that a $\K$-module $M$ is flat and
$\m$-adically complete if and only if $M \cong \what{N}$ for some free
$\K$-module $N$.

Sometimes one is given a ring homomorphism $\k \to \K$ lifting the canonical
surjection $\K \to \k$; and then $\K$ becomes a $\k$-algebra. In this case the
free $\K$-module $N$ can be expresses as $N = \K \otimes_{\k} V$ for some
$\k$-module $V$; and its completion 
is $M =  \what{N} = \K \hatotimes{\k} V$.
Moreover $V \cong \k \otimes_{\K} N \cong \k \otimes_{\K} M$ as $\k$-modules.

\begin{exa}
Take $\K := \k[[\hbar]]$, the power series ring in
the variable $\hbar$ over the field $\k$. The maximal ideal $\m$ is generated by
$\hbar$. For a $\k$-module $V$ we have a canonical isomorphism
$\k[[\hbar]] \hatotimes{\k} V \cong V[[\hbar]]$, the latter being set of formal
power series with coefficients in $V$.
\end{exa}

The next definition is used in \cite{Ye3}:

\begin{dfn}
Let $A$ be a flat $\m$-adically complete $\K$-algebra, such that the
$\k$-algebra
$\bar{A} := \k \otimes_{\K} A$ is commutative. 
Then we call $A$ an {\em associative $\K$-deformation of $\bar{A}$}.
\end{dfn}

If $\K$ is a $\k$-algebra then we can find a (noncanonical) isomorphism of 
$\K$-modules $A \cong \K \hatotimes{\k} \bar{A}$.
The multiplication induced on 
$\K \hatotimes{\k} \bar{A}$ by such an isomorphism is called a {\em star
product}.

\begin{exa}
Suppose $\bar{A}$ is some commutative $\k$-algebra, and
$\K = \k[[\hbar]]$. Then a star product $\star$ on the $\k[[\hbar]]$-module 
$A := \bar{A}[[\hbar]]$ is expressed by a series 
$\{ \beta_i \}_{i \geq 1}$ of $\k$-bilinear functions
$\beta_i : \bar{A} \times \bar{A} \to \bar{A}$, as follows:
\[ c_1 \star c_2 = c_1 c_2 + \sum_{i \geq 1} \beta_i(c_1, c_2) \hbar^i \]
for $c_1, c_2 \in \bar{A}$. 
\end{exa}

We shall need this version of the Nakayama Lemma: 

\begin{lem} \label{lem:4}
Let $\K$ be as in Setup \tup{\ref{setup:1}}, let $A$ be an $\m$-adically
complete $\K$-algebra, and let $M$ be a finitely generated left $A$-module. 
If $\k \otimes_{\K} M = 0$  then $M = 0$.
\end{lem}

\begin{proof}
Let $\a :=\m A$, which is a two-sided ideal of $A$, and
$\m^{i} A = \a^i$ for every $i$.  It follows that $A$ is $\a$-adically
complete. According to \cite[Section III.3.1, Lemma 3]{CA}
the ideal $\a$ is inside the Jacobson radical of $A$. By
the usual Nakayama Lemma (which holds also for noncommutative rings, cf.\
\cite[Section II.3.2, Proposition 4]{CA}) we see that 
$M / \a M = 0$ implies $M = 0$.
\end{proof}

Note that there is no commutativity or finiteness assumption on the algebra $A$;
only its structure as $\K$-module is important. 

The next proposition might be of interest.

\begin{prop}
Let $\K$ be as in Setup \tup{\ref{setup:1}}, let $A$ be an $\m$-adically
complete $\K$-algebra, and let $M$ be a perfect complex in
$\cat{D}(\cat{Mod} A)$. If 
$\k \otimes^{\mrm{L}}_{\K} M = 0$ then $M = 0$.
\end{prop}

\begin{proof}
Assume $M \neq 0$, and let $\mrm{H}^{i_0} M$ be its highest nonzero cohomology
module. By Lemmas \ref{lem:1} and \ref{lem:4} we see that
$\k \otimes_{\K} \mrm{H}^{i_0} M \neq 0$. On the other hand by the K\"{u}nneth
trick (Lemma \ref{lem:2}) we have
\[ \k \otimes_{\K} \mrm{H}^{i_0} M \cong
 \mrm{H}^{i_0} ( \k \otimes^{\mrm{L}}_{\K} M ) . \]
Hence $\k \otimes^{\mrm{L}}_{\K} M \neq 0$.
\end{proof}

Here is the main result of our paper:

\begin{thm} \label{thm:5}
Let $\K$ be as in Setup \tup{\ref{setup:1}}, and let $A$ and $B$ be a flat
$\m$-adically complete $\K$-algebras, such that the $\k$-algebras
$\bar{A} := \k \otimes_{\K} A$ and
$\bar{B} := \k \otimes_{\K} B$ are commutative with connected spectra. Suppose
$T$ is a two-sided tilting complex over
$B$-$A$ relative to $\K$. Then there is an isomorphism
\[ T \cong P[n] \]
in 
$\cat{D}^{\mrm{b}}(\cat{Mod} B \otimes_{\K} A^{\mrm{op}})$,
for some invertible $B$-$A$-bimodule $P$ and integer $n$.
\end{thm}

\begin{proof}
This is very similar to the proof of Theorem \ref{thm:2}. We may assume that
$A \neq 0$. Define
\[ n := - \opn{sup} \{ i \mid \mrm{H}^i T \neq 0 \} , \]
and let $P := \mrm{H}^{-n} T$. This is a $B$-$A$-bimodule. 
By Lemma \ref{lem:1}, $P$ is a nonzero finitely generated right $A$-module.
So according to Lemma \ref{lem:4} the right $\bar{A}$-module
$\bar{P} := \k \otimes_{\K} P$ is nonzero.
By the K\"{u}nneth trick (Lemma \ref{lem:2}) there is an isomorphism 
\[ \bar{P} = \k \otimes_{\K} \mrm{H}^{-n} T \cong
\mrm{H}^{-n} ( \k \otimes^{\mrm{L}}_{\K} T ) . \]

According to Corollary \ref{cor:1} the complex
$\bar{T} := \k \otimes^{\mrm{L}}_{\K} T$
is a two-sided tilting complex over
$\bar{B}$-$\bar{A}$ relative to $\k$.
Since $\bar{A}$ is commutative and $\opn{Spec} \bar{A}$ is connected, we can
apply Theorem \ref{thm:2}. The conclusion is that 
$\bar{T}$ has exactly one nonzero cohomology module. But by the calculation
above this must be 
$\mrm{H}^{-n} \bar{T} \cong \bar{P}$.
Therefore we get an isomorphism
$\bar{T} \cong \bar{P}[n]$ in 
$\cat{D}(\cat{Mod} \bar{B} \otimes_{\k} \bar{A}^{\mrm{op}})$,
and $\bar{P}$ is an invertible $\bar{B}$-$\bar{A}$ bimodule relative to $\k$.

Let 
$S \in \cat{D}^{\mrm{b}}(\cat{Mod} A \otimes_{\K} B^{\mrm{op}})$ 
be an inverse of $T$. 
Define
\[ m := - \opn{sup} \{ i \mid \mrm{H}^i S \neq 0 \} , \]
$Q := \mrm{H}^{-m} S$,
$\bar{S} := \k \otimes^{\mrm{L}}_{\K} S$
and
$\bar{Q} := \k \otimes_{\K} Q$. By the same considerations as above we see that
$\bar{S} \cong \bar{Q}[m]$ in 
$\cat{D}(\cat{Mod} \bar{A} \otimes_{\k} \bar{B}^{\mrm{op}})$,
and $\bar{Q}$ is an invertible $\bar{A}$-$\bar{B}$ bimodule relative to $\k$.

{} From Corollary \ref{cor:1} it follows that  
\[ \bar{P}[n] \otimes_{\bar{A}} \bar{Q}[m] \cong 
\bar{T} \otimes^{\mrm{L}}_{\bar{A}} \bar{S} \cong \bar{B} . \]
Therefore $n = -m$.
Using the K\"{u}nneth trick we see that
\[ B \cong \mrm{H}^0 (T \otimes^{\mrm{L}}_A S) \cong
(\mrm{H}^{-n} T) \otimes_A (\mrm{H}^{n} Q) = 
P \otimes_A Q . \]
Similarly we get
\[ A \cong Q \otimes_B P . \]
So $P$ is an invertible $B$-$A$-bimodule relative to $\K$.

Since $P$ is a projective $A^{\mrm{op}}$-module,
and it is the highest nonzero cohomology of $T$, 
by Lemma \ref{lem:1} we have an isomorphism
$T \cong M \oplus P[n]$
in $\cat{D}^{\mrm{b}}(\cat{Mod} A^{\mrm{op}})$ for some complex $M$. 
Suppose, for the sake of contradiction, that $M \neq 0$; and let
\[ l := \opn{sup} \{ i \mid \mrm{H}^i M \neq 0 \} . \]
Then $l < -n$, so $l + n < 0$.
By the K\"{u}nneth trick we get
\[ (\mrm{H}^{l} M) \otimes_A Q \cong 
(\mrm{H}^{l} M) \otimes_A (\mrm{H}^{n} S) \cong 
\mrm{H}^{l + n} (M \otimes^{\mrm{L}}_A S) , \]
which is a direct summand of the $B^{\mrm{op}}$-module
\[ \mrm{H}^{l + n} (T \otimes^{\mrm{L}}_A S) \cong
\mrm{H}^{l + n} B = 0 . \]
But $Q$ is an invertible bimodule, and therefore 
$\mrm{H}^{l} M = 0$. This is a contradiction.
Hence $T \cong P[n]$ in $\cat{D}^{\mrm{b}}(\cat{Mod} A^{\mrm{op}})$.

Finally, the last isomorphism implies that $\mrm{H}^i T = 0$ for all 
$i \neq -n$. By truncation we obtain the isomorphism 
$T \cong P[n]$ in 
$\cat{D}^{\mrm{b}}(\cat{Mod} B \otimes_{\K} A^{\mrm{op}})$.
\end{proof}

The upshot is that associative deformations behave like commutative algebras,
as far as derived Morita theory is concerned. Specifically:

\begin{cor} \label{cor:2}
Let $\K$ be as in Setup \tup{\ref{setup:1}}, and let $A$ and $B$ be a flat
$\m$-adically complete $\K$-algebras, such that the $\k$-algebras
$\bar{A} := \k \otimes_{\K} A$ and
$\bar{B} := \k \otimes_{\K} B$ are commutative with connected spectra. 
Assume that $A$ and $B$ are derived Morita equivalent relative to $\K$.
Then $A$ and $B$ are Morita equivalent relative to $\K$. Moreover the
$\k$-algebras $\bar{A}$ and $\bar{B}$ are isomorphic.
\end{cor}

\begin{proof}
By Theorem \ref{thm:1} there is a two-sided tilting complex
$T$ over $B$-$A$-relative to $\K$. 
Therefore by Theorem \ref{thm:5} there is an invertible 
$B$-$A$-bimodule $P$ relative to $\K$. So we have classical Morita equivalence
between $A$ and $B$.

Now the bimodule $\bar{P} := \k \otimes_{\K} P$ is an invertible 
$\bar{B}$-$\bar{A}$-bimodule relative to $\k$. Since these are commutative
$\k$-algebras they must be isomorphic.
\end{proof}

\begin{cor} \label{cor:3}
Let $\K$ be as in Setup \tup{\ref{setup:1}}, and let $A$ be a flat
$\m$-adically complete $\K$-algebra, such that the $\k$-algebra
$\bar{A} := \k \otimes_{\K} A$ is commutative with connected spectrum. 
Then 
\[ \opn{DPic}_{\K}(A) = \opn{Pic}_{\K}(A) \times \mbb{Z} . \]
\end{cor}

\begin{proof}
As mentioned earlier, there is a canonical inclusion
of $\opn{Pic}_{\K}(A) \times \mbb{Z}$ into 
$\opn{DPic}_{\K}(A)$. By Theorem \ref{thm:5} this is a bijection.
\end{proof}

\begin{rem}
Let $\K$ be any commutative ring, and let $A$ be a flat noetherian $\K$-algebra.
A {\em dualizing complex} over $A$ relative to $\K$ is a complex
$R \in \cat{D}^{\mrm{b}}(\cat{Mod} A \otimes_{\K} A^{\mrm{op}})$ 
satisfying a list of conditions; see \cite[Definition 4.1]{Ye1}.
Presumably \cite[Theorem 4.5]{Ye1} holds in this case (it was only proved when
$\K$ is a field). Then the group $\opn{DPic}_{\K}(A)$ classifies isomorphism
classes of dualizing complexes (if at least one dualizing complex exists).

Now assume we are in the situation of Corollary \ref{cor:3}, and that $\bar{A}$
is a finitely generated $\k$-algebra. Then $A$ is noetherian. It is reasonable
to suppose that $A$ will have some dualizing complex $R$ relative to $\K$. What 
Corollary \ref{cor:3} tells us is that any other dualizing complex $R'$ must be
isomorphic to $P[n] \otimes_A R$ for some invertible bimodule $P$ and integer
$n$.
\end{rem}

\begin{rem}
In the paper \cite{BW} Bursztyn and Waldmann consider the local ring
$\K = \k[[\hbar]]$, and a fixed commutative $\k$-algebra $\bar{A}$
with connected spectrum. 
They prove that the Picard group
$\opn{Pic}_{\k}(\bar{A})$ acts on the set of gauge equivalence classes of 
associative $\K$-deformations $A$ of $\bar{A}$. The orbit of a deformation $A$
under this action is the set of deformations that Morita equivalent to $A$.
The stabilizer of $A$ in $\opn{Pic}_{\k}(\bar{A})$ is the image of 
$\opn{Pic}_{\K}(A)$.
And the kernel of the homomorphism
$\opn{Pic}_{\K}(A) \to \opn{Pic}_{\k}(\bar{A})$
is the group of outer gauge equivalences of $A$.

Presumably these results remain true for any complete ring $\K$ as in Setup
\ref{setup:1}, not just for $\K = \k[[\hbar]]$.
\end{rem}


\end{document}